\documentclass[oneside,english]{amsart}
\usepackage[T1]{fontenc}
\usepackage[latin9]{inputenc}
\usepackage{amsthm}
\usepackage{amssymb}
\usepackage{graphics}
\usepackage{graphicx}
\usepackage{epsfig}
\usepackage{epstopdf}
\usepackage{morefloats}
\usepackage{marginnote}
\usepackage{hyperref}

\usepackage{babel}
\usepackage{color}

\makeatletter

\numberwithin{equation}{section}
\numberwithin{figure}{section}
\theoremstyle{plain}
\newtheorem{thm}{Theorem}
  \theoremstyle{plain}
  \newtheorem{lem}[thm]{Lemma}
  \newtheorem{prop}[thm]{Proposition}

\makeatletter

\makeatletter

\makeatletter

\usepackage{latexsym}

\makeatletter
\AtBeginDocument{
  
}

\makeatother

\AtBeginDocument{

}

\makeatother
\makeatother
\makeatother

\AtBeginDocument{
  
}

\makeatother

\begin{document}

\title{Exponential growth rates of free and amalgamated products}

\author{Michelle Bucher, Alexey Talambutsa}

\address{Universit\'e de Gen\`eve, Section de Math\'ematiques, 2-4 rue du Li\`evre, \linebreak Case postale~64, 1211 Gen\`eve 4, Suisse}
\email{michelle.bucher-karlsson@unige.ch}

\address{Steklov Mathematical Institute of RAS, Department of mathematical logic, \linebreak Gubkina~8, 119991 Moscow, Russia}
\email{altal@mi.ras.ru}

\begin{abstract}
We prove that there is a gap between $\sqrt{2}$ and $(1+\sqrt{5})/2$ for the exponential growth rate of free products $G=A*B$ not isomorphic to the infinite dihedral group. For amalgamated products $G=A*_C B$ with $([A:C]-1)([B:C]-1)\geq2$, we show that lower exponential growth rate than $\sqrt{2}$ can be achieved by proving that the exponential growth rate of the amalgamated product $\mathrm{PGL}(2,\mathbb{Z})\cong (C_2\times C_2) *_{C_2} D_6$ is equal to the unique positive root of the polynomial $z^3-z-1$. This answers two questions by Avinoam Mann [The growth of free products, Journal of Algebra 326, no. 1 (2011) 208--217]. \end{abstract}

\thanks{Michelle Bucher was supported by Swiss National Science Foundation
project PP00P2-128309/1. Alexey Talambutsa was supported by Russian Foundation for Basic Research, project 11-01-12111.
The authors thank the Institute Mittag-Leffler in Djursholm, Sweden,
for their warm hospitality during  the preparation of this paper.}
\maketitle

\section{Introduction}

Let $G$ be a group generated by a finite set $S$. For any element $g\in G$ the length $\ell_{G,S}(g)$ is defined as
the minimal possible integer $n$ such that 
$g=x_1x_2\dots x_n$, where $x_i\in S\cup S^{-1}$ for all $i\in \{1,2,\ldots,n\}$. The corresponding growth function $F_{G,S}(n)$ counts the number of elements $g\in G$ for which $\ell_{G,S}(g)\leqslant n$. The \emph{exponential growth rate} of $G$ with respect to $S$ is the limit
$$\omega(G,S)=\lim\limits_{n\to \infty}(F_{G,S}(n))^{\frac1n},$$ which always exists by subadditivity of the growth function. The value of $\omega(G,S)$ may depend on the choice of generating set $S$. It is obviously always $\omega(G,S)\geq 1$ and there is hence an infimum $\Omega(G)=\inf \{\omega(G,S)\}$ which is an invariant of the group $G$.
If however the strict inequality $\omega(G,S)>1$ is true for some generating set $S$ then $\omega(G,S')>1$ for any generating set $S'$, and the group $G$ has \emph{exponential growth}. If further $\Omega(G)>1$ then $G$ has  \emph{uniform exponential growth}.

Following Wilson's original example of a group of exponential growth with $\Omega(G)=1$, there are now series of such examples (see \cite{Wilson,Bartholdi,Nekrashevych}). Still, these constructions are somewhat exceptional, and for many classes of groups of exponential growth such as linear groups \cite{EMO}, hyperbolic groups \cite{Koubi}, solvable groups \cite{Osin}, amalgamated products and HNN extensions \cite{dlHBucher}, one relator groups \cite{GrigorchukDlH}, etc. it is proved that $\Omega(G)>1$. We will concentrate on free products and amalgamated products for which uniform exponential growth was established in \cite{dlHBucher} and address here the question of the sharpness of lower bounds for the minimal exponential growth rate for these groups.

For free products $G=A*B$ not isomorphic to the infinite dihedral group, it is not difficult to show that $\Omega(A*B)\geq \sqrt{2}$ \cite[Theorem 4]{Mann}, which is a sharp inequality since $\Omega(C_2*C_3)=\sqrt{2}$, where $C_k$ denotes the cyclic group of order $k$. The next example, computed by Mann, shows  $\Omega(C_2*C_4)=\frac{1+\sqrt5}2$ \cite[Theorem 6]{Mann}. We negatively answer the question of Mann \cite[Problem 2]{Mann} whether there exists  a free product $A*B$ for which $\sqrt2<\Omega(A*B)<\frac{1+\sqrt5}2$.

\begin{thm}
\label{thm: free product}Let $G=A*B$ be the free product of the
groups $A$ and $B$. If $G$ is not isomorphic to $C_2*C_2$ or $C_2*C_3$, then $\Omega(G)\geq\frac{1+\sqrt{5}}{2}$.
\end{thm}

Based on inequalities of Lyons, Pichot and Vassout between exponential growth rate and $\ell^2$-Betti numbers, Theorem \ref{thm: free product} was known for all but a finite number of groups. Indeed, it is shown in \cite{Lyons} that
$$\Omega(A*B)\geq 1+2\beta^{(2)}_1(A * B) = 3+2\beta^{(2)}_1(A) + 2\beta^{(2)}_1(B)  - 2/|A| - 2/|B|,$$
which is greater or equal to $5/3$ unless the order of $A$ (or $B$) is $2$ and the order of $B$ (or $A$) is $2,3,4$ or $5$. The minimal growth rate of $C_2*C_5$ is computed in \cite{Talambutsa2011} and is strictly greater than the golden ratio (and further also greater than $5/3$). Thus, we note that the only new case covered by Theorem  \ref{thm: free product} is $G=C_2 * (C_2 \times C_2)$, which, as a group generated by elements of order $2$, precisely falls in the category of groups for which the methods in \cite{Mann} fail. We present a unified and direct geometric proof of the general Theorem \ref{thm: free product}.



For amalgamated products $G=A*_C B$ with $([A:C]-1)([B:C]-1)\geq2$, the rough inequality $\Omega(A*_C B)\geq \sqrt[4]2$ was proved by the first author and de la Harpe \cite{dlHBucher} in order to establish that these groups are of uniform exponential growth. It is asked in \cite[Problem 3.3]{dlHGrowth} if the lower bound $\sqrt[4]2$ can be improved to $\sqrt{2}$. This is the case if the subgroup $C$ is normal in $A$ and $B$, in which case $\Omega(A*_C B)=\Omega(A/C * B/C)$ and in particular $$\Omega(\mathrm{PSL}(2,\mathbb{Z}))= \Omega(\mathrm{SL}(2,\mathbb{Z}))= \sqrt{2}.$$
Mann raised the lower bound of $ \sqrt[4]2$ to the so called \emph{plastic number} $\alpha$, which is the unique positive root of the polynomial $z^3-z-1$ and is equal to the cubic irrationality
$$
\alpha = \sqrt[3]{\frac12+\frac16\sqrt{\frac{23}3}}+\sqrt[3]{\frac12-\frac16\sqrt{\frac{23}3}}.
$$
This number appears in number theory and other notable occasions (see \cite{wiki-Plastic}).

\begin{thm}[Mann, Theorem 2 in \cite{Mann}]
\label{thm: amalgamated product}Let $G=A*_{C}B$ be the amalgamated
product of the groups $A$ and $B$ over $C$. If $([A:C]-1)([B:C]-1)\geq2$
then $\Omega(G)\geq  \alpha.$
\end{thm}

We provide here a geometric proof of Mann's lower bound. The question whether there exists a group as in Theorem \ref{thm: amalgamated product} with $\Omega(G)<\sqrt 2$ was again raised by Mann \cite[Problem 3]{Mann}. We answer this question positively by showing in Proposition \ref{prop: example amalg} that the minimal exponential growth rates of $\mathrm{PGL}(2,\mathbb{Z})\cong(C_2\times C_2)*_{C_2} D_6$ and consequently also of $\mathrm{GL}(2,\mathbb{Z})\cong D_8 *_{(C_2\times C_2)} (D_6\times C_2)$, where $D_{2k}$ denotes the group of symmetries of a regular $k$-gon, are
$$\Omega(\mathrm{PGL}(2,\mathbb{Z}))=\Omega(\mathrm{GL}(2,\mathbb{Z}))=\alpha,$$
thus proving that the lower bound of Theorem \ref{thm: amalgamated product} is sharp.

\section{Exponential growth rate}




\begin{lem}
\label{lemma: x y xy when x y free monoid}
Let $G$ be a group and $S$ a finite generating set. Let $x,y\in S$. Suppose that the element $t=xy$ also belongs to $S$.
If $\langle x,y\rangle^{+}$ is a positive free monoid, then \[
\omega(G,S)\geq\left(\frac{1+\sqrt{5}}{2}\right)^{2}.\]

\end{lem}

\begin{proof}
Consider the set $\mathcal{P}$ of positive words in the letters $x,y,t$ which do not contain $xy$ as a subword. First we note that two different words $U,V$ from $\mathcal{P}$ represent different elements of $G$. Indeed, if $U$ and $V$ end with the same letter, we may cancel it and proceed by induction on length. Otherwise, after rewriting $U,V$ in the letters $x,y$, their 2-letter endings will be different, so $U\ne V$ in $G$.

Now let $X(n),Y(n),T(n)$ be the number of words in $\mathcal{P}$ which end with the letters $x,y,t$ respectively. Also set $W(n)=X(n)+Y(n)+T(n)$.
Then the following system of relations hold true
\begin{eqnarray*}
X(n+1)&=&X(n)+Y(n)+T(n)=W(n),\\
Y(n+1)&=&Y(n)+T(n)=W(n)-X(n),\\
T(n+1)&=&X(n)+Y(n)+T(n)=W(n).
\end{eqnarray*}
Summing these equalities, we get the recurrent relation $W(n+1)=3W(n)-X(n)=3W(n)-W(n-1)$ with a characteristic polynomial $x^2-3x+1$ that has roots $\alpha_1=(3+\sqrt 5)/2=((1+\sqrt 5 )/2)^2$ and $\alpha_2=(3-\sqrt 5)/2$. As an order $2$ linear homogeneous recurrent relation, the series $W(n)$ can be presented in the form $W(n)=c_1\alpha_1^n+c_2\alpha_2^n$ for some constants $c_1,c_2$. Since the set $\mathcal{P}$ contains any word in the letters $x$,$t$, the series $W(n)$ has an exponential growth with exponent bigger than $2$. Because $\alpha_2<1$ and the summand $c_2\alpha_2^n$ is exponentially decreasing, it follows that the coefficient $c_1$ is necessarily positive . Finally, since $f_{G,S}(n)\geqslant W(n)$, we have $\omega(G,S)\geqslant \lim\limits_{n\to \infty}\sqrt[n]{W(n)}=\alpha_1$, which finished the proof.
\end{proof}

The following lemma will be used below in many cases to compute lower bounds for the growth rate of the groups in terms of growth rates of positive free submonoids. Interestingly, the growth rate of some positive submonoids realize the growth rate of the groups in all minimal cases.

\begin{lem} \label{lem: lengthMonoid} Let $G$ be a group generated by a finite set $S$. Suppose that there exist $x_1,\dots,x_k\in G$ generating a positive free monoid inside $G$. Set $\ell_i:=\ell_S(x_i)$, for $i=1,\ldots,k$, and $m=\max\{ \ell_1,\dots,\ell_k\}$. Then $\omega(G,S)$ is greater or equal to the unique positive root of the polynomial
\begin{equation}
 Q(z)=z^m-\sum_{i=1}^k z^{m-\ell_i}.
 \label{char-polynomial}
 \end{equation}
\label{lem: computation for positive monoid}
\end{lem}

\begin{proof}
Let $X_1,\ldots, X_k$ be the words in the alphabet $S\cup S^{-1}$ that give expressions for $x_1,\ldots,x_k$ of length $\ell_1,\ldots,\ell_k$. We will consider the set $\mathcal{P}$ consisting of all words that are made of the pieces $X_1,\ldots,X_k$. Let $T_i(n)$ be the number of words in $\mathcal{P}$ that have length $n$ and end with a piece $X_i$, let $W(n)$ be the sum $T_1(n)+\ldots+T_k(n)$. Then obviously $T_i(n)=W(n-\ell_i)$. This leads to a recurrent relation
$$
W(n)=W(n-\ell_1)+\ldots+W(n-\ell_k).
$$
This relation has the polynomial $Q(z)$ from \eqref{char-polynomial} as the characteristic polynomial and the generating function $R(z)=\sum_{k=0}^{\infty} W(k) z^k$ can be presented as $P(z)/Q(z)$ for some integer polynomial $P(z)$. Since $R(z)$ has non-negative coefficients and also $\lim_{n\to \infty}\sqrt[n]{W(n)}=d\geq 1$, the radius of convergence for $R(z)$ is equal to $1/d$ and according to \cite[Theorem 1.8.1]{Bieberbach} it is reciprocal to one of the positive roots of the characteristic polynomial $Q(z)$. According to Descartes' rule of signs, there is only one such root $z_1$ and hence we get $d=z_1$. This leads to the equality $\lim\sqrt[n]{W(n)}=z_1$ and then because $F_{G,S}(n) \geqslant W(n)$ we get the claimed inequality $\omega(G,S)\geqslant z_1$.
\end{proof}

We finish this section with some simple lower bounds for the growth rate of a free product of cyclic groups with respect to a non canonical generating set.

\begin{lem} \label{lem: free prod of cyclic, non canonical basis} Let $G=C_m*C_n=\langle a,b\mid a^m=b^n=1\rangle\ne C_2*C_2$ and $0<k<m$. Then
$$\omega(C_m*C_n,\{a,ba^{-k}\}) \geq \frac{1+\sqrt{5}}{2}.$$
\end{lem}

\begin{proof} Suppose that $m\geq 3$. Upon replacing $a$ by $a^{-1}$  and consequently $k$ by $m-k$ we can assume that $k+1<m$. Consider the elements $x=ba^{-k}$ and $y=ba^{-k-1}=(ba^{-k})a^{-1}$ which have length $1$ and $2$ in our generators  $a,ba^{-k}$. One can easily see from the normal forms that $\langle x,y \rangle^+$ generate a free monoid. Hence, it follows from Lemma \ref{lem: computation for positive monoid} that the growth $\omega(G,S)$ is greater or equal to the unique positive root of  $z^2-z-1$, which is equal to the golden ratio.

Suppose now that $m=2$ and in particular that $k=1$. If $n\geq 4$, then consider the elements $x=ba$, $y=b^2a=(ba)a(ba)$ and $z=b^{-1}a=a(ba)^{-1}a$. Again, from the normal form it follows that $\langle x,y,z \rangle^+$ generate a free monoid. As $x,y,z$ have length $1,3,3$ the growth of $G$ is greater or equal to the positive root of $z^3-z^2-2$ by Lemma \ref{lem: computation for positive monoid}, which is strictly greater than the golden ratio. For $n=3$, the exact growth rate of golden ratio has been computed by Mach\`i (see \cite[VI.A.7 (ii)]{dlHBook}) and we repeat the argument for the convenience of the reader. Writing $t=ba$ we obtain the presentation
$$
C_2*C_3=\langle a,t \mid a^2=1, tat=at^{-1}a \rangle.
$$

Let $W$ be a word of minimal length in $a,t$ representing $x\in C_2*C_3$, that also has minimal possible number of $t$-letters. Since $a$ has order $2$, we can assume that $W$ contains no subwords of the form $a^m$, for $m\neq 1$. Thus $W$ has the form
\begin{equation}
W=a^{\nu_0}t^{r_1}at^{r_2}a\cdot \dots \cdot at^{r_\ell}a^{\nu_1},
\label{at-normalform}
\end{equation}
where the $r_i$'s are nonzero integers and $\nu_0$ and $\nu_1$ are either $0$ or $1$. Moreover, the signs of $r_i$ and $r_{i+1}$ are opposite for $i\in \{1,\ldots,\ell-1\}$ since otherwise we would be able to use the relations $tat=at^{-1}a$ and $t^{-1}at^{-1}=ata$ to reduce the number of $t$-letters. To show that distinct words of such form represent distinct elements of the group we rewrite the word $W$ in the letters $a$ and $b$, substituting $t$ by $ba$. It is easy to see that in the case when $r_1$ is positive and $r_\ell$ is negative we will get a word of the form
\begin{equation}
a^{\nu_0} \underbrace{b a b \ldots a b}_{r_1 \text{ letters } b} a \underbrace{b^{-1} a b^{-1} \ldots a b^{-1}}_{-r_2 \text{ letters } b^{-1}} a \ldots a \underbrace{b^{-1} a b^{-1} \ldots a b^{-1}}_{-r_{\ell} \text{ letters } b^{-1}} a^{\nu_1}.
\label{comb-normalform}
\end{equation}
From a word of this form we can uniquely determine the numbers $\nu_0,\nu_1,r_1,\ldots,r_{\ell}$ and then find the initial word $W$. If the signs of $r_1$ and $r_{\ell}$ are not $(+,-)$, we can rewrite $W$ in a similar form to \eqref{comb-normalform} and the initial form \eqref{at-normalform} can also be determined from the subsequences of letters of $b$ and $b^{-1}$.


It remains to count the number of words, and for this, we denote by $A(n)$ and $T(n)$ the number of words of length $n$ ending in $a$ and $t$ or $t^{-1}$ respectively. Set $W(n)=A(n)+T(n)$. We obtain the recurrent relation
$$\begin{array}{rcl}
A(n)&=&T(n-1),\\
T(n)&=&T(n-1)+A(n-1),
\end{array}
$$
and thus $W(n)=W(n-1)+W(n-2)$, giving the claimed growth rate of golden ratio.
\end{proof}

\section{Group actions on trees}

Let $T$ be a simplicial tree and $x\in \mathrm{Aut}(T)$ be a nontrivial automorphism of $T$. For simplicity, we will restrict to non edge reversing automorphisms of the tree (thus, if $x$ fixes an edge $e$, then it fixes each of the vertices of $e$). Note that an edge reversing automorphism can always be made non edge reversing upon replacing $T$ by its first barycentric subdivision. Define
$$\tau(x)=\min\{ d(v,x(v))\mid v\in T^0 \}$$
as the {\it translation distance} of $x$, where $d:T^0\times T^0\rightarrow \mathbb{N}$ denotes the simplicial distance on the vertex set $T^0$. We will say that $x$ is  {\it elliptic} if $\tau(x)=0$ and  {\it  hyperbolic} if $\tau(x)>0$. Observe that $x$ is elliptic if and only if its fixed point set
$$\mathrm{Fix}(x)=\{ v\in T^0 \mid x(v)=v \}$$
is non empty. Since if $x$ fixes two vertices $v$ and $w$ it also fixes the geodesic segment between $v$ and $w$, the fixed point set $\mathrm{Fix}(x)$ is always connected. If $x$ is hyperbolic, there exists a unique invariant bi-infinite geodesic line $L_x$ in $T$, called the {\it axis} of $x$, such that $x$ acts on $L_x$ by translation by $\tau(x)$. In particular, $d(v,x(v))=\tau(x)$ if and only if $v$ belongs to $L_x$. 

\begin{lem}\label{lem: criterion hyperbolic} Let $x\in \mathrm{Aut}(T)$. Suppose that there exists a subset $\{v_1,v_2,w_1,w_2\}\subset T^0$ of cardinality at least $3$ on a geodesic segment in $T$, such that $x(v_1)=w_1$, $x(v_2)=w_2$ and $d(v_1,w_1)=d(v_2,w_2)>0$. Then $x$ is hyperbolic, $v_1,v_2,w_1,w_2$ belong to $L_x$ and $\tau(x)=d(v_1,w_1)=d(v_2,w_2)$.
\end{lem}

\begin{proof} Let $v\in T^0$ be any vertex. We claim that the midpoint $m$ between $v$ and $x(v)$ is either a fixed point if $x$ is elliptic or belongs to the axis $L_x$ in case $x$ is hyperbolic. Suppose that $x$ is elliptic and suppose that $v\neq x(v)$. Let $\gamma=[v,w]$ be the geodesic segment from $v$ to $\mathrm{Fix}(x)$. The elliptic element $x$ maps $\gamma$ to a geodesic segment whose intersection with $\gamma$ is precisely $w$. In particular the distance $d(v,x(v))$ is equal to twice the length of $\gamma$ and $w\in \mathrm{Fix}(x)$ is the midpoint between $v$ and $x(v)$. Suppose now that $x$ is hyperbolic. There is a unique geodesic segment $\gamma=[v,w]$ (possibly reduced to a vertex) between $v$ and the axis $L_x$, where $w\in L_x$ is the closest vertex to $v$. The hyperbolic element will translate $w\in L_x$ by $\tau(x)$ and map the segment $\gamma$ to a segment $x(\gamma)$ whose intersection with $L_x$ is precisely $x(w)$. Thus, the vertex $v$ has been moved by $d(v,x(v))=2d(v,L_x)+\tau(x)$. The midpoint between $v$ and $x(v)$ is obviously equal to the midpoint between $w$ and $x(w)$ and thus lies on the axis $L_x$. In particular, the intersection of $L_x$ with the geodesic segment from $v$ to $x(v)$ is an interval of length $\tau(x)$ centered on the midpoint between $v$ and $x(v)$.

Let us now come back to the set up of the lemma. Let $\gamma$ be the geodesic segment containing $v_1,v_2,w_1,w_2$. Let $m_1,m_2$ be the midpoints between $v_1,w_1=x(v_1)$ and $v_2,w_2=x(v_2)$ respectively. Note that since $\{v_1,v_2,w_1,w_2\}$ contains at least three different vertices on a geodesic segment, the midpoints $m_1,m_2$ are distinct and of course also lie on $\gamma$.

Suppose that $x$ were elliptic. Then the segment $\gamma_m$ between $m_1$ and $m_2$ would be contained in the fixed point set of $x$. Note that none of the four points $v_1,v_2,w_1,w_2$ are fixed points, and in particular they cannot lie on $\gamma_m$. Thus, the points $v_1,w_1=x(v_1)$, and similarly $v_2,w_2$, lie in the two different connected components of $\gamma \setminus \gamma_m$. Upon replacing $x$ by its inverse, we can suppose that $v_1$ is closer to $m_2$ than to $m_1$. We compute
$$d(v_1,x(v_1))=2d(v_1,\mathrm{Fix}(x))\leq 2d(v_1,\gamma_m)=2d(v_1,m_2)<2d(v_1,m_1)=d(v_1,w_1),$$
a contradiction.

We have thus established that $x$ is hyperbolic. Suppose that $v_1$ or $v_2$ does not lie on $L_x$. Since $d(v_1,x(v_1))=2d(v_1,L_x)+\tau(x)=d(v_2,x(v_2))$, all four points $v_1,v_2,w_1,w_2$ lie at the same distance from $L_x$ and can thus impossibly lie on a geodesic segment since the cardinality of the set $\{v_1,v_2,w_1,w_2\}$ is at least $3$.
\end{proof}

\begin{lem}
\label{lem: x,y elliptic => xy hyperbolic}Let $x,y$ be elliptic.
If $\mbox{Fix}(x)\cap \mbox{Fix}(y)=\emptyset$, then $xy^{-1}$ is hyperbolic of translation length $\tau(xy^{-1})=2\cdot d(\mbox{Fix}(x),\mbox{Fix}(y))$ and
the axis $L_{xy^{-1}}$ contains the segment between $\mbox{Fix}(x)$ and
$\mbox{Fix}(y)$ and its image under both $x$ and $y$.
\end{lem}

\begin{proof} Let $v_x\in \mbox{Fix}(x)$ and $v_y\in \mbox{Fix}(y)$ be the endpoints of the geodesic segment between $\mbox{Fix}(x)$ and $\mbox{Fix}(y)$. Since the intersection of the geodesic segment $[v_x,v_y]$ with the geodesic segment $x[v_x,v_y]=[v_x,x(v_y)]$ contains only the vertex $v_x$, the three vertices $v_y,v_x,x(v_y)$ lie on a geodesic segment, and similarly for $v_x,v_y,y(v_x)$. It follows that the four vertices $y(v_x),v_y,v_x,x(v_y)$ lie on a geodesic segment,
$$xy^{-1}(y(v_x))=v_x, \  xy^{-1}(v_y)=x(v_y)\ \mathrm{and} \ d(y(v_x),v_x)=d(v_y,x(v_y))=2\cdot d(v_x,v_y).$$
Applying Lemma \ref{lem: criterion hyperbolic}, it is immediate that $xy^{-1}$ is hyperbolic, that its axis contains the four vertices $y(v_x),v_y,v_x,x(v_y)$ and that its translation length is equal to
$$\tau(xy^{-1})=2\cdot d(v_y,v_x)=2\cdot d(\mbox{Fix}(x),\mbox{Fix}(y)),$$
as claimed.
\end{proof}

\medskip
\begin{center}
\includegraphics[width=90mm]{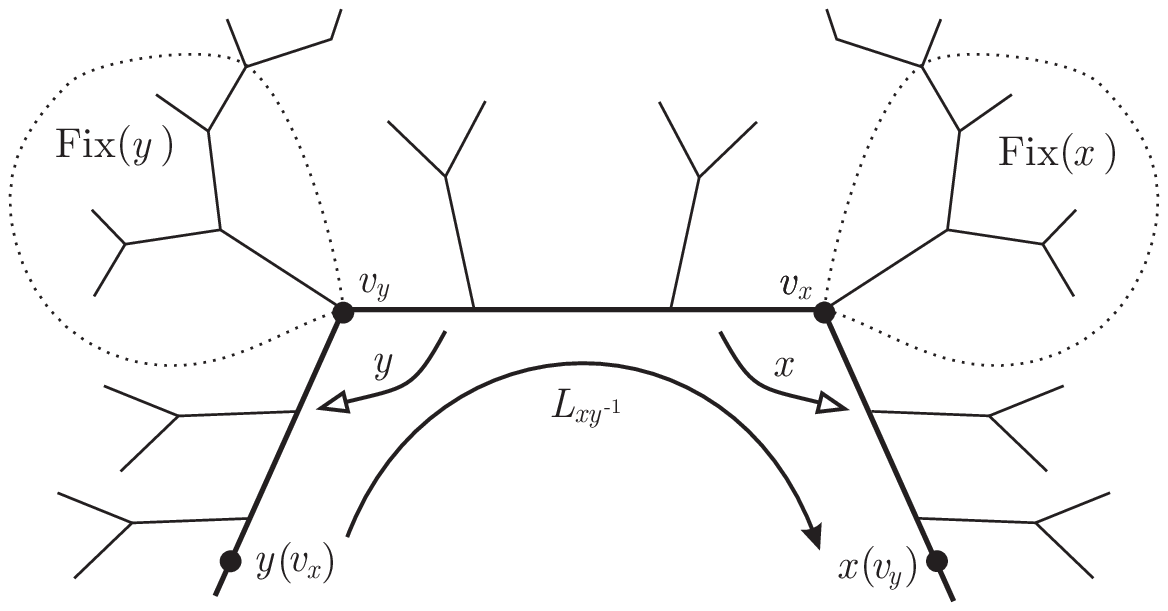}
\end{center}
\medskip

\begin{lem} \label{lem:x(L)=L} Let $x,y$ be elliptic with $\mbox{Fix}(x)\cap \mbox{Fix}(y)=\emptyset$. If $L_{xy}=L_{x^{-1}y}=L_{xy^{-1}}=L_{x^{-1}y^{-1}}$, then $x(L_{xy})=L_{xy}$.
\end{lem}

\begin{proof} Keeping the notation of Lemma \ref{lem: x,y elliptic => xy hyperbolic}, the axis of the hyperbolic element $xy^{-1}$ contains the vertices
$$ \dots   ,yx^{-1}yx^{-1}(v_y),\ yx^{-1}(v_y),\ v_y,\ x(v_y),\ xy^{-1}x(v_y),\ xy^{-1}xy^{-1}x(v_y)\dots$$
Since the axis of $xy$, $x^{-1}y$, $xy^{-1}$ and $x^{-1}y^{-1}$ all agree, and by the proof of Lemma \ref{lem: x,y elliptic => xy hyperbolic}, we know that all these elements act as the same translation on their common axis, it follows that the four hyperbolic elements all act on the above ordered set of vertices by translation to the right by $1$. We deduce that for any $\varepsilon_0\in \{\pm 1,0\}$ and $\varepsilon_1,...,\varepsilon_{2k}\in \{\pm 1\}$, the vertex $w=x^{\varepsilon_0}y^{\varepsilon_1}\cdot \dots \cdot  x^{\varepsilon_{2k}}(v_y)$ belongs to the common axis (and is independent of the signs of the $\varepsilon_i$'s), and in particular, also $x(w)$. Since this holds for infinitely many vertices on $L_{xy}$, the lemma is proven.
\end{proof}

\subsection*{Ping-Pong and free (semi-)groups} One classical way to find free subgroups in groups going back to Klein's study of Schottky groups is through so called Ping-Pong arguments. We recall here the statements, both in the subgroup and semi(sub)group case, and refer to \cite[Proposition 1.1]{Tits} and \cite[Proposition VII.2]{dlHBook} respectively for the proofs. The latter reference treats the case $k=2$ of Lemma \ref{lem: PingPongMonoid}, but the generalization to arbitrary $k$ is straightforward.

\begin{lem}[Ping-Pong Lemma] \label{lem: PingPong} Let $G$ be a group acting on a set $Z$. Let $G_X,G_Y$ be two subgroups of $G$. Suppose that there exists subsets $X,Y\subset Z$ and a base point $z\in Z\setminus (X\cup Y)$ such that
$$\begin{array}{ll}
x(Y\cup \{v\} )\subset X & \forall x\in G_X\setminus \{\mathrm{Id}\}, \\
y(X\cup \{v\} )\subset Y & \forall y\in G_Y\setminus \{\mathrm{Id}\}.
\end{array}
$$
Then the subgroup $\langle G_X,G_Y \rangle$ is a free product $G_X*G_Y$.
\end{lem}

\begin{lem}[Ping-Pong Lemma for positive free monoids] \label{lem: PingPongMonoid} Let $G$ be a group acting on a set $X$. Let $x_1,\ldots,x_k\in G$ and $X_1,\ldots,X_k$ be nonempty subsets of $X$. Suppose that $X_i\cap X_j=\emptyset$ whenever $i\neq j$ and that
$$
x_i(X_1\cup\dots \cup X_k)\subset X_i, \quad \forall i\in \{ 1,\dots,k\}.
$$
Then the semigroup generated by $x_1,\ldots,x_k$ is a positive free monoid.
\end{lem}

\begin{lem}
\label{lem: Intersection axis}Let $x,y$ be hyperbolic. Suppose that the length of the intersection of the axis $L_{x}\cap L_{y}$
is strictly smaller than the minimum of the translation lengths of
$x$ and $y$, then $\langle x,y\rangle\cong F_{2}$.
\end{lem}

\begin{proof} Let $\gamma$ be either the geodesic segment between $L_x$ and $L_y$ (if $L_x\cap L_y=\emptyset$) or the intersection   $L_x\cap L_y$. Let $X$, respectively $Y$, be the intersection of the vertex set $T^0$ with the two connected components of $T\setminus \gamma$ containing $L_x\setminus \gamma$, resp. $L_y\setminus \gamma$.

\medskip
\begin{center}
\includegraphics[width=120mm]{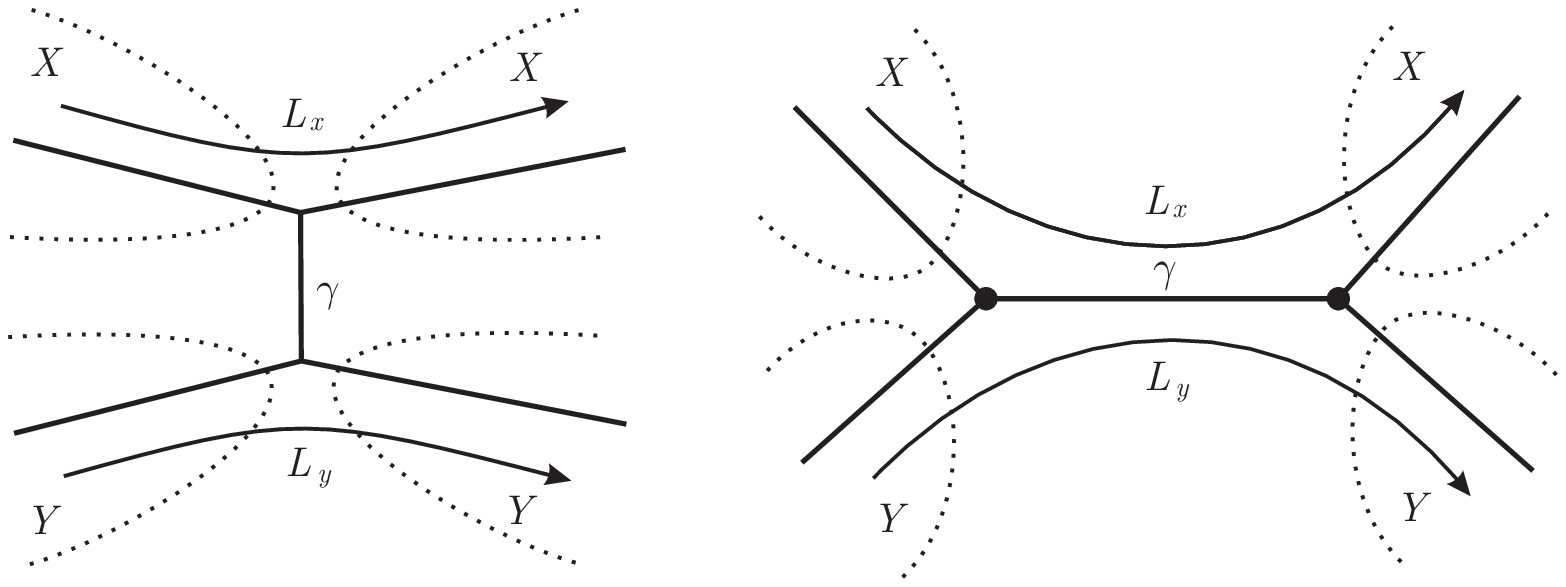}
\end{center}
\medskip
Pick a vertex $v\in \gamma\cap T^0$. Observe that $x^k(Y\cup \{ v\})\subset X$ and $y^k(X\cup \{ v\})\subset Y)$ for any $k\in \mathbb{Z}\setminus \{ 0\}$. Thus the conclusion follows from the Ping-Pong Lemma \ref{lem: PingPong}.
\end{proof}

\begin{lem} \label{lem: different axis pos free monoid} Let $x,y$ be hyperbolic. If $L_x\neq L_y$ then either $x,y$ generate a positive free monoid or $x^{-1},y$ do.
\end{lem}

\begin{proof} If the intersection $L_x\cap L_y$ is empty or equal to one vertex, then the claim follows from Lemma \ref{lem: Intersection axis}. Otherwise it contains a geodesic segment $\gamma$ (possibly infinite on one end). Let $v$ be an end of $\gamma$. Upon replacing $x$ by $x^{-1}$ we can suppose that $d(v,x(v))=d(\gamma,x(v))$ (that is to say, $x$ moves $v$ away from $\gamma$). Similarly for $y$.

Let $v_x$, respectively $v_y$, be the vertices on $L_x$, resp. $L_y$, at distance $1$ from $v$ and not on $\gamma$. Set
$$X=\{w\in T^0\mid d(w,v)>d(w,v_x)\}$$
and
$$Y=\{w\in T^0\mid d(w,v)>d(w,v_y)\}.$$
Observe that $x(X\cup Y)\subset X$ and $y(X\cup Y)\subset Y$. The conclusion now follows from the Ping-Pong Lemma \ref{lem: PingPongMonoid} for positive monoids.
\end{proof}

\subsection*{Example: amalgamated products}

Let $G=A*_C B$ be an amalgamated product. Define the associated tree of $G$ as follows: The vertices of $T$ are the right cosets of $G$ by $A$ and $B$:
$$T^0=G/A \amalg G/B.$$
The edges are the right cosets of $G$ by $C$:
$$T^1=G/C.$$
The vertices of the edge $gC$, for $g\in G$, are $gA$ and $gB$. Observe that the tree is bipartite. The group $G$ acts on $T$ by left multiplication. An element $x\in G< \mathrm{Aut}(T)$ is elliptic if and only if $x$ is conjugated to $A$ or $B$ (or in particular to $C$). If $G$ is a free product, i.e. $C$ is trivial, then the fixed point set of any elliptic element consists of a single vertex.

\subsection*{Subgroups of free products}

\begin{lem}\label{lem: xi yi free prod}
Let $G=A*B$ be a free product. Let $x_1,\dots,x_n,y_1,\dots,y_m\in G$ be elliptic elements. Suppose that $\mbox{Fix}(x_1)=\dots =\mbox{Fix}(x_n)\neq\mbox{Fix}(y_1)=\dots =\mbox{Fix}(y_m)$, then  \[
\langle x_1,\dots, x_n,y_1,\dots, y_m \rangle=\langle  x_1,\dots, x_n\rangle*\langle y_1,\dots, y_m\rangle.\]
\end{lem}

\begin{proof} Let $\gamma$ be  the interior of the geodesic segment between $\mbox{Fix}(x_1)=\dots =\mbox{Fix}(x_n)$ and $\mbox{Fix}(y_1)=\dots =\mbox{Fix}(y_m)$. Let $X$, respectively $Y$, be the connected component of $T\setminus \gamma$ containing $\mbox{Fix}(x_1)=\dots =\mbox{Fix}(x_n)$, resp. $\mbox{Fix}(y_1)=\dots =\mbox{Fix}(y_m)$. Pick $v\in \gamma$.

\medskip
\begin{center}
\includegraphics[width=80mm]{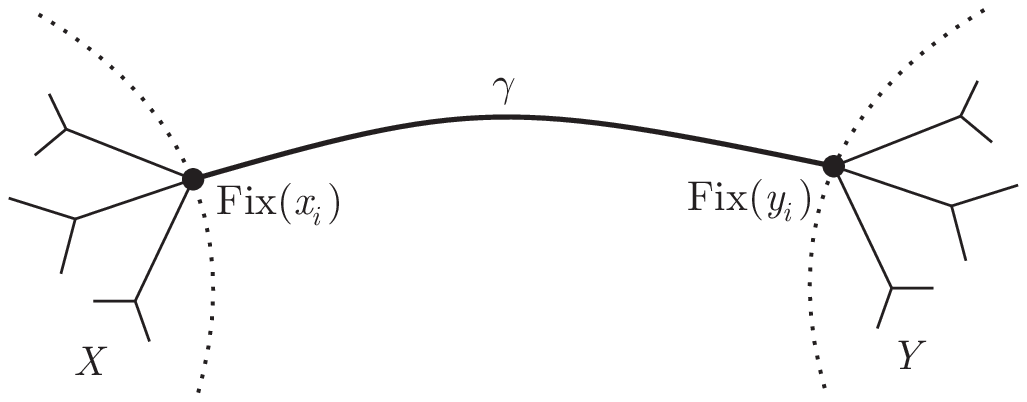}
\end{center}
\medskip

Observe that $x(Y\cup \{ v\})\subset X$ for any $x\in \langle  x_1,\dots, x_n\rangle \setminus \{ \mathrm{id} \} $ and $y(X\cup \{ v\})\subset Y)$ for any $y\in \langle y_1,\dots, y_m\rangle \setminus \{ \mathrm{Id} \}$. Thus the conclusion follows from the Ping-Pong Lemma \ref{lem: PingPong}. \end{proof}

\begin{lem}\label{lem: x ell y hyp}
Let $G=A*B$ be a free product. Let $x$ be elliptic and $y$ hyperbolic. If $\mbox{Fix}(x)\notin L_y$, then \[
\langle x,y\rangle=\langle x\rangle*\langle y\rangle.\]
\end{lem}

\begin{proof} Let $\gamma$ be  the interior of the geodesic segment between $\mbox{Fix}(x)$ and $L_y$. Let $X$, respectively $Y$, be the union of the connected components $C$ of $T\setminus \gamma$ for which $d(C,\mbox{Fix}(x)) = d(C,\gamma)$, resp. $d(C,L_y)\leq d(C,\gamma)$. Pick $v$ on $\gamma$. Observe that $x^k(Y\cup \{ v\})\subset X$ for any $k$ not a multiple of the order of $x$, and $y^k(X\cup \{ v\})\subset Y$ for any $k\in \mathbb{Z}\setminus \{ 0\}$. Once again, the conclusion follows from the Ping-Pong Lemma \ref{lem: PingPong}.
\end{proof}

\begin{prop}\label{prop: gp gen by hyp and ell} Let $G=A*B$ be a free product. Let $x$ be elliptic and $y$ hyperbolic. Then there exists $\ell \in \mathbb{Z}$ such that
$$\langle x,y\rangle=\langle x\rangle*\langle y x^\ell \rangle.$$
\end{prop}

\begin{proof} If $\mbox{Fix}(x)\notin L_y$, then the conclusion follows from Lemma \ref{lem: x ell y hyp} (with $\ell=0$). Suppose that $\mbox{Fix}(x)\cap L_y=\{v_x\}$ and let $m$ be the midpoint between $v_x$ and $y(v_x)$ and $m'$ be the midpoint between $v_x$ and $y^{-1}(v_x)$. (Note that $m$ and $m'$ are vertices of $T$ since the translation length of $y$ is an even integer.)

Suppose that there exists $\ell \in \mathbb{Z}$ such that $x^\ell (m)=m'$. Then $m$ is a fixed point of $yx^\ell$ which is distinct from the fixed point $v_x$ of $x$. The conclusion follows by applying Lemma \ref{lem: xi yi free prod} to the two elliptic elements $x$ and $yx^\ell$.

If $x^\ell(m)\neq m'$ for any $\ell\in \mathbb{Z}$, set
$$X=\{ w\in T^0 \mid d(w,L_y)=d(w,(m',m))\}$$
and
$$Y=\{ w\in T^0 \mid d(w,L_y) = d(w,L_y \setminus (m',m))\}\setminus \{m\},$$
where $(m',m)$ denotes the open interval from $m'$ to $m$.

\medskip
\begin{center}
\includegraphics[width=80mm]{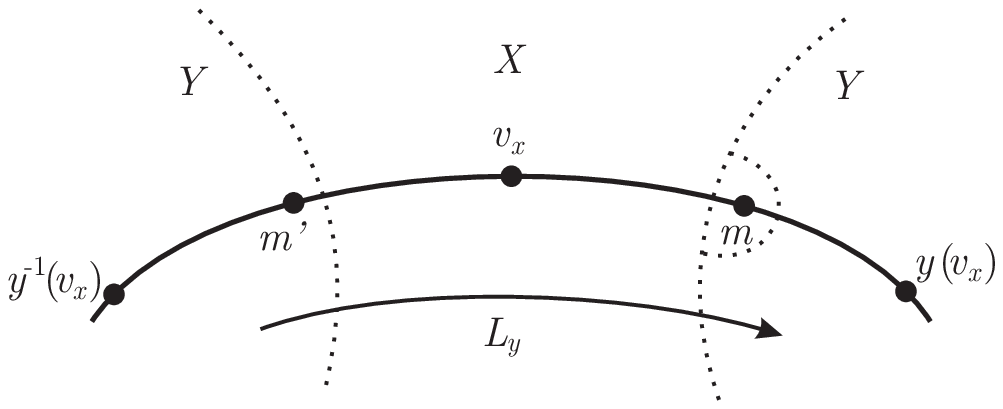}
\end{center}
\medskip

Note that $x^k(Y\cup \{m\} )\subset X$ for any $k$ not a multiple of the order of $x$. Indeed, this is because the only vertices at distance equal to $\tau(y)/2$ from $v_x$ which are not in $X$ are $m$ and $m'$, and as $x^k(m)\neq \{m,m'\}$, the element $x^k$ maps $m$ and $m'$, and consequently also $Y$ into $X$. We also have the obvious inclusion $y^k(X\cup \{m\})\subset Y$ for any $k\neq 0$. Apply the Ping-Pong Lemma \ref{lem: PingPong} to $m,X,Y$ to conclude that
$$\langle x,y\rangle=\langle x\rangle*\langle y \rangle,$$
as desired.
\end{proof}

\section{Proof of Theorem \ref{thm: free product}}

Let $G=A*B$ be a free product and $S$ a finite generating set. We start with preliminary estimates for the growth of $G$ with respect to $S$ in the case when $S$ contains three elliptic elements with distinct fixed points.

\begin{lem} \label{lem: 3 ell} Let $x,y,z\in S$ be elliptic elements with distinct fixed points. If $\langle x,y,z \rangle \neq C_2*C_2$ then $\omega(G,S)\geq\frac{1+\sqrt{5}}{2}$.
\end{lem}

\begin{proof} Let $v_x,v_y,v_z\in T^0$ be the fixed points of $x,y,z$ respectively. We distinguish two cases: Either $v_x,v_y,v_z$ lie on a geodesic segments, or they do not, in which case they lie at the extremity of a tripod.

Tripod case: Let $m\in T^0$ be the midpoint of the tripod. Observe by Lemma \ref{lem: x,y elliptic => xy hyperbolic} that the axis $L_{xy^{-1}}$ of $xy^{-1}$ passes through $v_x, v_y$ and $y(v_x)$, while the axis  $L_{yz^{-1}}$ of $yz^{-1}$ passes through $v_z, v_y$ and $y(v_z)$. It follows that the intersection $L_{xy^{-1}}\cap L_{yz^{-1}}$ is equal to the geodesic segment from $m$ to $y(m)$ and in particular the length of the intersection is equal to $2\cdot d(v_y,m)$.

\medskip
\begin{center}
\includegraphics[width=90mm]{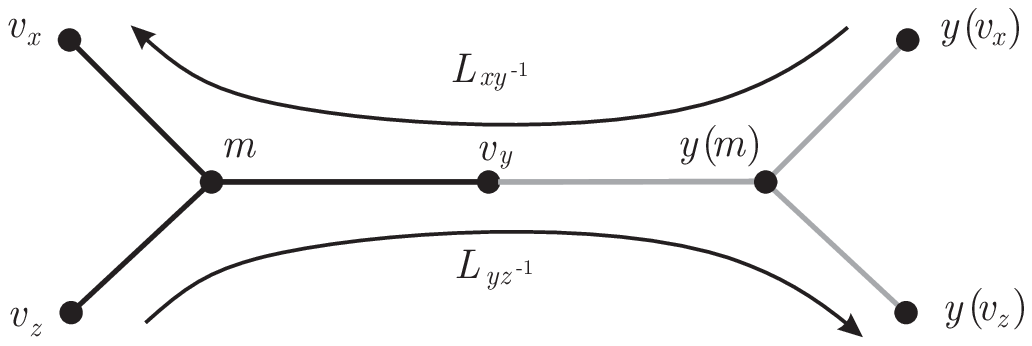}
\end{center}
\medskip

The translation lengths $2\cdot d(v_x,v_y)$,  respectively $2\cdot d(v_y,v_z)$, of $xy^{-1}$ and $yz^{-1}$ being strictly bigger than the length $2\cdot d(v_y,m)$ of the intersection $L_{xy^{-1}}\cap L_{yz^{-1}}$, it follows from Lemma \ref{lem: Intersection axis} that the two hyperbolic elements $xy^{-1}$ and $yz^{-1}$ generate a free subgroup. We conclude by Lemma \ref{lemma: x y xy when x y free monoid} that
$$\omega(G,S)\geq \sqrt{\omega(\langle xy^{-1},yz^{-1}\rangle,\{xy^{-1},yz^{-1},xz^{-1}\} )}\geq \frac{1+\sqrt{5}}{2}.$$
Observe that to obtain the bound of Lemma \ref{lemma: x y xy when x y free monoid} we only need a positive free monoid. As we have here a free group, it is possible to improve the present lower bound to $2$.

Geodesic case: By symmetry, we suppose that $v_y$ belongs to the geodesic segment from $v_x$ to $v_z$. Let $w_x$, respectively $w_z$ be the vertex at distance one from the vertex $v_y$ on the geodesic segment towards $v_x$, respectively $v_z$. If $z(w_x)\neq w_z$ or $z(w_z)\neq w_x$ we can by symmetry suppose that $y(w_z)\neq w_x$. Set
\begin{eqnarray*}
X&=&\{ w\in T^0 \mid d(w,w_x)\leq d(w,v_y) \},\\
Z&=&\{ w\in T^0 \mid d(w,w_z)\leq d(w,v_y) \}.
\end{eqnarray*}
Since the axis of $yz^{-1}$ passes through $v_z,v_y$ and $y(w_z)$ it is obvious that $yz^{-1}(X\cup Z)\subset Z$. Similarly, the axis of $xy^{-1}$ passes through $y(w_x),v_y$ and $v_x$. Since $y(w_x)\neq y(w_z)$, we also have $xy^{-1}(X\cup Z)\subset X$  so that by Lemma \ref{lem: PingPongMonoid}, the two hyperbolic elements $xy^{-1}$ and $yz^{-1}$ generate a free subgroup, and we conclude as in the tripod case that $\omega(G,S)\geq (1+\sqrt{5})/2$.

\medskip
\begin{center}
\includegraphics[width=100mm]{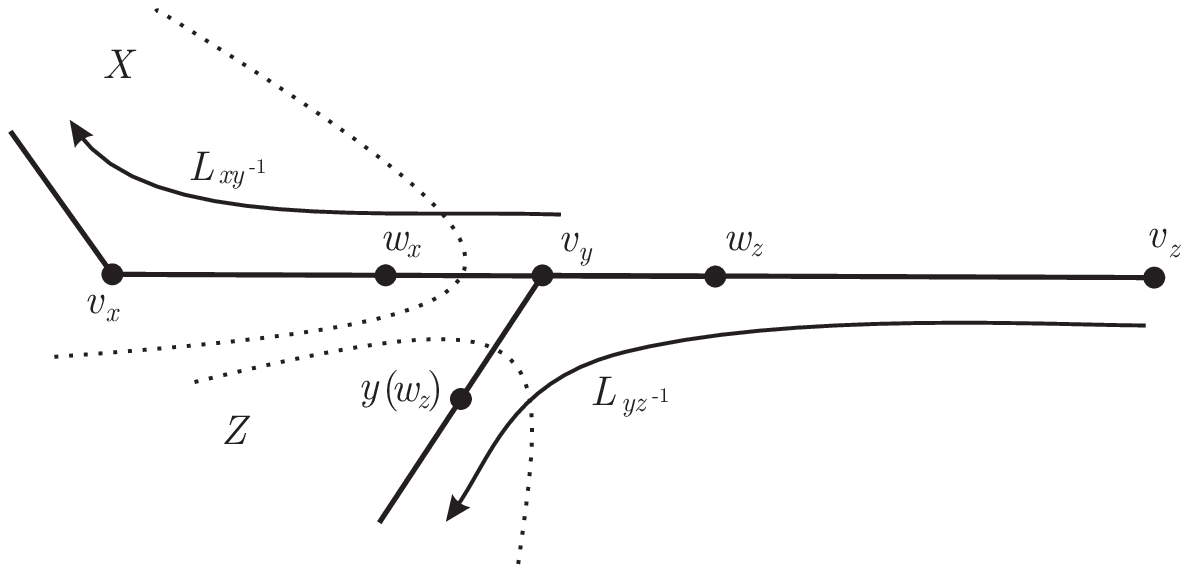}
\end{center}
\medskip

If $y(w_x)=w_z$ and $y(w_z)=w_x$ then $y^2$ fixes the three vertices $w_x,v_y,w_z$ and is hence the identity, so that $y$ has order $2$. If $L_{xy}\neq L_{yz^{-1}}$ we conclude similarly as above that $xy^{-1}$ and $yz^{-1}$ generate a positive monoid giving the desired lower bound for the growth. Note that in this case the intersection of the two axis contains the segment between $w_x$ and $w_z$ and it is important that the translation direction is the same for $xy^{-1}$ and $yz^{-1}$. Thus we can suppose that $L_{xy}=L_{yz^{-1}}$ and upon replacing $x$ and $z$ by their inverses, also that $L:=L_{xy}=L_{x^{-1}y}=L_{yz}=L_{yz^{-1}}$. Lemma \ref{lem:x(L)=L} then implies that $x(L)=y(L)=z(L)=L$. We thus get a homomorphism $\langle x,y,z \rangle \rightarrow \mathrm{Aut}(L)\cong C_2*C_2$. This homomorphism is an injection since a free product action cannot fix more than one vertex without being trivial. It follows that  $\langle x,y,z \rangle\cong C_2*C_2$ as the group generated by $x,y,z$ is a nontrivial free product and can hence neither be the trivial group, nor $C_2$ nor $\mathbb{Z}$.
\end{proof}

\begin{lem} \label{lem: 3 ell dihedral} Let $S$ be a finite generating set for $G=A*B$ consisting only of elliptic elements. Suppose that the set of fixed points $\{ v\in T^0\mid x(v)=v \ \mathrm{for \ some \ } x\in S\}$ has cardinality at least equal to $3$. If $G\neq C_2 *C_2$, then there exists $x,y,z\in S$ with distinct fixed points such that $\langle x,y,z\rangle \neq C_2*C_2$.
\end{lem}

\begin{proof} Let $x,y,z\in S$ have distinct fixed points. Note that as $\langle x,y,z \rangle \cong C_2*C_2$ the three elements $x,y,z$ have order $2$. If two among the axis $L_{xy},L_{yz},L_{zx}$ were different, then the group generated by $x,y,z$ would have exponential growth, which is impossible. Thus $L:=L_{xy}=L_{yz}=L_{zx}$ and by Lemma \ref{lem:x(L)=L} we have $x(L)=y(L)=z(L)=L$. Now let $z'\in S$ be another elliptic element. By symmetry we can suppose that $x,y,z'$ have distinct fixed points. As previously, we conclude that $L=L_{xy}=L_{yz'}$ and $z'(L)=L$. Thus, the whole group $G$ preserves the infinite geodesic $L$ and we have a surjective map $G\rightarrow C_2*C_2$ which is also an injection since an element of the free product fixes a whole line (and in particular two vertices of its tree) only if it is the identity. This contradicts our assumption that $G\neq C_2*C_2$.
\end{proof}

\begin{proof}[Proof of Theorem \ref{thm: free product}] Let $S$ be a generating set for the free product $G=A*B$. 


If there exist two hyperbolic elements $y,y'\in S$ with different axis $L_y\neq L_{y'}$, then the subgroup they generate contains a positive free monoid and the growth rate $\omega(G,S)$ is greater or equal to $2$. So we suppose that all hyperbolic elements in $S$ have the same axis $L_y$. Since $G$ does not preserve $L_y$, there must exist an elliptic element $x\in S$ with $x(L_y)\neq L_y$. By Proposition \ref{prop: gp gen by hyp and ell} the group generated by $x$ and $y$ is isomorphic to a free product $\langle x \rangle * \langle yx^\ell \rangle$ for some $\ell \in \mathbb{N}$ strictly smaller than the order of $x$. If $\ell=0$ the growth rate is greater or equal to $2$. If $\ell\neq 0$, the growth rate is, by Lemma \ref{lem: free prod of cyclic, non canonical basis} greater or equal to the golden ratio, unless $\langle x,y\rangle =\langle x \rangle * \langle yx \rangle=C_2*C_2$. But in the latter case, both $yx$ and $x$ have order $2$, so the hypothesis of Lemma \ref{lem:x(L)=L} for the two hyperbolic elements $yx$, $y$ with disjoint fixed points is trivially satisfied, so that $x(L_{(yx)x})=x(L_y)=L_y$, contradicting our hypothesis.




We can from now on suppose that $S$ does not contain
any hyperbolic elements, or in other words, that all $x\in S$ are
elliptic. Suppose that the set of fixed points $\{ v\in T^0\mid x(v)=v \ \mathrm{for \ some \ } x\in S\}$ has cardinality at least equal to $3$. Since $G\neq C_2*C_2$, then there exist $x,y,z\in S$ with distinct fixed points such that $\langle x,y,z\rangle \neq C_2*C_2$ by Lemma \ref{lem: 3 ell dihedral}. Hence $\omega(G,S)\geq (1+\sqrt{5})/2$ by Lemmas \ref{lem: 3 ell}. We can thus restrict to the case where there are two possible fixed point sets: $\mbox{Fix}(x_1)=\dots=\mbox{Fix}(x_n)\neq \mbox{Fix}(y_1)=\dots=\mbox{Fix}(y_m)$ for $S=\{x_1,\dots,x_n,y_1,\dots,y_m\}$. (Note that the case of a single fixed point set is excluded by the fact that $G$ has no global fixed point.) By Lemma \ref{lem: xi yi free prod}, $G$ is isomorphic to the free product
$$G=\langle x_1,\dots x_n\rangle * \langle y_1,\dots y_m \rangle.$$

If the sets $\{x_1^{\pm 1},\dots x_n^{\pm 1}\}$ and $\{y_1^{\pm 1},\dots y_m^{\pm 1}\}$ both contain at least two elements $x,x'$ and $y,y'$ respectively, then counting the distinct words obtained by alternating one of $x,x'$ with one of $y,y'$ already shows $\omega(G,S)\geq 2$.

We can thus suppose that $n=1$ and $x_1=x_1^{-1}=:x$. If $\{y_1^{\pm 1},\dots y_m^{\pm 1}\}$ contains three elements $y,y',y''$ then considering the distinct words obtained by alternating $x$ with one of $y,y',y''$ gives $\omega(G,S)\geq \sqrt{3}$. We can thus restrict to the case when $\{y_1^{\pm 1},\dots y_m^{\pm 1}\}$ has two elements, which means that either $m=2$ with $y_1^2=y_2^2=1$ or $m=1$ with $y_1\neq y_1^{-1}$. (The case when the set contains exactly one element, so that $m=1$ with $y_1^2=1$ is excluded by the fact that $G$ is not the infinite dihedral group.)

Suppose that $m=2$ and $y_1^2=y_2^2=1$. If $y_1y_2\neq y_2 y_1$ then we count the distinct words $W(n)$ of length $n$ in the alphabet $\{x,y_1,y_2\}$ obtained by alternating $x$ with one of $y_1,y_2,y_1y_2$ or $y_2y_1$. This leads to the recurrent relation
$$W(n) = 2W(n-2)+2W(n-3).$$
The growth rate of $G$ is then bigger than the unique positive root of $x^3-2x-2$, which is strictly greater than the golden ratio. If $y_1y_2=y_2y_1$ then $\langle y_1,y_2\rangle=C_2\times C_2$ and the growth rate is equal to the golden ratio. Indeed, counting the distinct words obtained by alternating $x$ with one of $y_1,y_2,y_1y_2$ leads to the recurrent relation
$$W(n)=2W(n-2)+W(n-3)$$
and the corresponding polynomial
$$x^3-2x-1=(x+1)(x^2-x-1)$$
that has a single positive root $(1+\sqrt{5})/2$.

Finally, suppose that $m=1$. Set $y:=y_1$ and let $k\geq 4$ be the order of $y$. (If $k=2$ or $3$ then $G\cong C_2*C_2$ or $C_2*C_3$.) If $k\geq 5$ then we count the words $W(n)$ of length $n$ obtained by alternating $x$ with one of $y,y^{-1},y^2,y^{-2}$ which gives the same recurrent relation $W(n)=2W(n-2)+2W(n-3)$ as above and growth strictly greater than golden ratio. If $k=4$, we alternate between $x$ and one of $y,y^{-1},y^2$ giving the relation $W(n)=2W(n-2)+W(n-3)$ and growth rate equal to the golden ratio.

\end{proof}

\section{Proof of Theorem \ref{thm: amalgamated product}}

Let $S$ be a generating set for the amalgamated product $G=A*_{C}B$.
The assumption $([A:C]-1)([B:C]-1)\geq2$ tells us that the corresponding tree is not a point nor a line, and in particular $G$ does not fix a point or preserve a line.

Suppose that $S$ contains a hyperbolic element $y$. Since $G$ cannot preserve the axis $L_y$, there must exist $x\in S$ with $x(L_y)\neq L_y$. It follows that the group generated by the two hyperbolic elements $y$ and $xyx^{-1}$ contains a free positive monoid, so that, by Lemma \ref{lem: lengthMonoid}, the growth rate $\Omega(G,S)$ is bigger or equal to the unique positive real root of
$$f(z)=z^3-z^2-1,$$
which is strictly greater than $\sqrt{2}$ since $f(\sqrt{2})<0$ while $\lim_{z\rightarrow +\infty}f(z)=+\infty$. Note that in this case an argument of Avinoam Mann \cite[Proposition 8]{Mann} would allow to improve this lower bound to $(1+\sqrt{5})/2$, based on the fact that $(xyx^{-1})^k$ does not have length $3k$ as counted in our estimate but length $k+2$.

We can from now on suppose that $S$ does not contain any hyperbolic
elements, or equivalently, that all $x\in S$ are elliptic. Since $G$ does not fix a vertex, the intersection $\cap_{x\in S}\mbox{Fix}(x)$
of all the fixed point sets is empty. It follows that there exist $x,y\in S$ such that $\mbox{Fix}(x)\cap\mbox{Fix}(y)=\emptyset$. Note that by Lemma \ref{lem: x,y elliptic => xy hyperbolic}, the product $xy$ is hyperbolic.

Suppose that there exists $x',y'\in S\cup S^{-1}$ such that $\mbox{Fix}(x')\cap\mbox{Fix}(y')=\emptyset$ and $L_{x'y'}\neq L_{xy}$. Then, by Lemma \ref{lem: Intersection axis}, the subgroup $\langle xy, x'y' \rangle$ of $G$ contains a free positive monoid and\[
\Omega(G,S)\geq\Omega(\langle x,y,x',y'\rangle,\{x,y,x',y'\})\geq\sqrt{\Omega(\langle xy,x'y'\rangle,\{xy,x'y'\}}\geq\sqrt{2}.\]
We can thus from now on suppose that for every  $x',y'\in S\cup S^{-1}$ such that $\mbox{Fix}(x')\cap\mbox{Fix}(y')=\emptyset$ the axis of $x'y'$ is equal to $L_{xy}$. This applies in particular to the inverses of $x'$ and $y'$,  and in particular implies that
$$L_{x'y'}=L_{x'^{-1}y'}=L_{x'y'^{-1}}=L_{x'^{-1}y'^{-1}}=L_{xy},$$
from which we deduce $x'(L_{xy})=L_{xy}$ by Lemma \ref{lem:x(L)=L}.

Since $G$ does not preserve a geodesic, there must exist $z\in S$ such that $z(L_{xy})\neq L_{xy}$. Thus, the intersections of the fixed point sets $\mbox{Fix}(x)\cap\mbox{Fix}(z)$ and $\mbox{Fix}(y)\cap\mbox{Fix}(z)$ cannot be empty. Since $\mbox{Fix}(z)$ is connected, $z$ must fix the segment between $\mbox{Fix}(x)\cap L_{xy}$ and $\mbox{Fix}(y)\cap L_{xy}$. In particular, $z$ fixes an edge $e$ of $L_{xy}$.

We claim that $\langle xy, zxy \rangle $ contains a free semigroup. Indeed, there is a nonempty proper subsegment (possibly infinite on one side) of $L_{xy}$ fixed by $z$. Let $v$ be a vertex at one of the extremity of the fixed segment. Upon replacing $xy$ by $y^{-1}x^{-1}$ we can suppose that $xy$ moves $v$ away from $\mbox{Fix}(z)\cap L_{xy}$.

Observe that the points $v_1=(zxy)^{-1}(v)=(xy)^{-1}(v)$, $w_1=v_2=v$, $w_2=zxy(v)$ lie on a geodesic and $d(v_1,w_1)=d(v_2,w_2)=\tau(xy)$. Thus, by Lemma \ref{lem: criterion hyperbolic}, the element $zxy$ is hyperbolic and its axis is different from $L_{xy}$ since $zxy(v)$ lies on $L_{zxy}$ but not on $L_{xy}$.
\medskip
\begin{center}
\includegraphics[width=65mm]{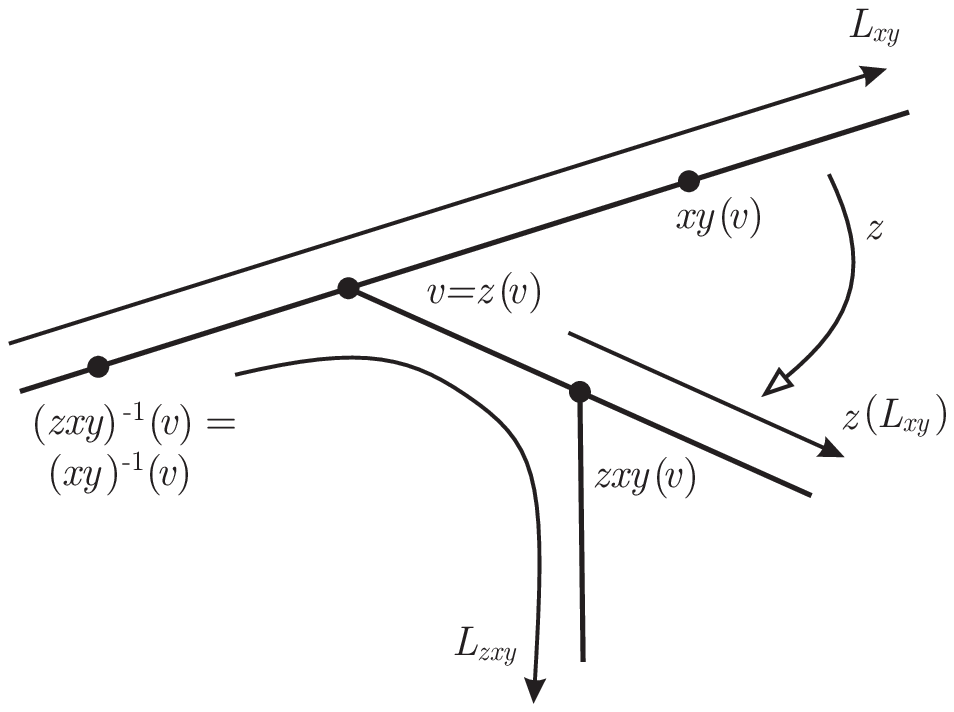}
\end{center}
\medskip

It follows that the group generated by $xy$ and $zxy$ contains a positive free monoid and consequently, by Lemma \ref{lem: lengthMonoid}, that the growth rate of $G$ is greater or equal to the unique positive root of
$$z^3-z-1,$$
which finishes the proof of the theorem. \qed


\section{Examples}

Let $\alpha$ be the unique positive root of the polynomial $z^3-z-1$. Theorem \ref{thm: amalgamated product} has shown that the minimal growth rate of any amalgamated product $A*_C B$ is bigger or equal to $\alpha$, provided that $[A:C]\geqslant 2,  [B:C]\geqslant 3$. We will prove that this lower bound is sharp for the group $\mathrm{PGL}(2,\mathbb{Z})$ that has the well-known decomposition $(C_2\times C_2)*_{C_2} D_6$. Consequently we will construct an infinite class of amalgamated free products having this growth rate.

\smallskip


First we consider the groups $C_2\times C_2 = \langle a,b \mid a^2=b^2=1, [a,b]=1\rangle$ and $D_6=\langle c,d \mid c^2=d^2=1, cdc=dcd \rangle$ and amalgamate them over the subgroups $\langle a\rangle$ and $\langle d\rangle$, both isomorphic to $C_2$. Then a presentation for $G=(C_2\times C_2)*_{C_2} D_6$ is
$$
G=\langle a,b,c,d \mid a^2=1,b^2=1, [a,b]=1, c^2=1, d^2=1, cdc=dcd, a=d \rangle.
$$
Removing the generator $d$, we obtain the presentation
\begin{equation*}
G=\langle a,b,c \mid a^2=1,b^2=1, [a,b]=1, c^2=1, cac=aca \rangle.
\end{equation*}
As $a$,$b$,$c$ are involutions, we may rewrite this presentation as in Proposition \ref{prop: example amalg}, and we get a well-known presentation for the group $\mathrm{PGL}(2,\mathbb{Z})$ (see \cite[formula 7.24]{Coxeter-Moser}).

\begin{prop} \label{prop: example amalg} The exponential growth rate of the group
\begin{equation}
G=\langle a,b,c \mid a^2=1,b^2=1, c^2=1, (ab)^2=1, (ac)^3=1 \rangle\cong \mathrm{PGL}(2,\mathbb{Z}),
\label{PGL-presentation}
\end{equation}
with respect to $S=\{a,b,c\}$ is equal to
$$\omega(G,S)=\alpha,$$
where $\alpha$ is the unique positive root of the polynomial $z^3-z-1$.
\end{prop}

Before the proof of the proposition, we start by showing how to write each element $x\in G$ in a unique special form of minimal length.

\begin{lem}
For any element $x\in G$ there exist a unique positive word $W(a,b,c)$ of length equal to $\ell_{G,S}(x)$ and of the form $U(a,b)$ or
\begin{equation}
\begin{array}{ccccccccccccc}
            & &  \nearrow    & b  & \searrow    &   & \nearrow    & b  & \searrow &    &  &      & \\
  U(a,b)  & c &              &    &             & c &             &    &          & c &\cdots& c & V(a,b), \\
            & &  \searrow    & ab & \nearrow    &   & \searrow    & ab & \nearrow &    &  &      &
            \label{plastic-form}
\end{array}
\end{equation}
where $U(a,b)$ and $V(a,b)$ are either one of the words $a,b,ab$ or the empty word.
\label{lem: normal form}
\end{lem}

\begin{proof} Let $X_0$ be any word representing the element $x$ of length equal to $\ell_{G,S}(x)$. As $a,b,c$ are involutions,
the word $X_0$ can be rewritten as a strictly positive word $X_1$. Moreover, as $a$ and $b$ commute,
we may rewrite $X_1$ to a word $X_2$ which has only occurrences of $ab$, but none of $ba$.
As $a^2=b^2=c^2=1$, there is no similar adjacent letters in $X_2$. If there are no $c$-letters in the word $X_2$, it is either empty or is one of the words $a,b,ab$. Otherwise, the word $X_2$ has to be written in the form
\[
\begin{array}{ccccccccccccc}
            & &  \nearrow    & b  & \searrow    &   & \nearrow    & b  & \searrow &  &   &      & \\
  U(a,b)  & c &  \rightarrow & a  & \rightarrow & c & \rightarrow & a  & \rightarrow & c &\cdots& c & V(a,b), \\
            & &  \searrow    & ab & \nearrow    &   & \searrow    & ab & \nearrow &  &   &      &
\end{array}
\label{abc-normalform}
\]
where the words $U(a,b)$ and $V(a,b)$ are as in the statement of the lemma.

Among all the words of this form representing $x$ and of length $\ell_{G,S}(x)$, there is at least one word $X_3$ which has the smallest possible number of letters $c$. Such a word does not contain a subword $cac$, as this subword could be rewritten as $aca$ giving a smaller number of $c$-letters in $X_3$. From this we immediately obtain that $X_3$ has the form \eqref{plastic-form}.

To prove the uniqueness we take two different words $W_1,W_2$ of the form \eqref{plastic-form} and consider the word $W_1W_2^{-1}$. We will use a standard argument about normal forms of amalgamated free products (see \cite[Theorem IV.2.6]{Lyndon-Schupp}) to show that $W_1W_2^{-1}\ne 1$. We may suppose that $U_1(a,b)\ne U_2(a,b)$ and $V_1(a,b)\ne V_2(a,b)$ as otherwise we can reduce $W_1$ and $W_2$ and proceed by induction on the length. So now we have
$$R=W_1W_2^{-1}=U_1(a,b) \ldots c \cdot V_1(a,b) V_2^{-1}(a,b) \cdot c \ldots U_2^{-1}(a,b).$$
Conjugating by $b$ if necessary we make $R$ start and end with either $a$ or $ab$ syllables. If $c\cdot V_1 V_2^{-1} \cdot =cac$ then we replace this occurrence by $aca$ and see that $R$ becomes a non-empty word of the form \eqref{plastic-form} which consists of syllables $b,ab$ separated by some letters $c$, so it is a reduced normal form, hence $R \ne 1$, and the uniqueness is proved. If one of the words $W_1$ or $W_2$ has form $U(a,b)$, the proof is identical.
\end{proof}

\begin{proof}[Proof of Proposition \ref{prop: example amalg}]
The lower bound $\Omega(G)\geq \alpha$ follows from Theorem \ref{thm: amalgamated product}. To get the upper bound, we will estimate $\omega(G,S)$ for the generating set $S=\{a,b,c\}$ by counting the number of words having form \eqref{plastic-form}. For this, we note the following recurrence relations between the number $W(n)$ of all words of length $n$ of the form \eqref{plastic-form} and $C(n)$ the number of those words that end with a $c$-letter:
\begin{equation*}
\begin{array}{rcl}
W(n) &=&C(n)+C(n-1)+C(n-2),\\
C(n) &=& C(n-2)+C(n-3).
\end{array}
\end{equation*}
We now easily see that the second equation of this system has $\alpha$ as a growth exponent for the recurrent sequence,
hence the first one (as a linear combination of the second) has the same exponent. Since both equations are linear recurrences,
the function $W(n)$ has a rational generating function $P(z)/Q(z)$. The growth function $f_{G,S}(n)=W(0)+W(1)+\ldots+W(n)$ has
generating function $P(z)/((1-z)Q(z))$ that obviously has the same radius of convergence $1/\alpha$, hence $\Omega(G)\leq\omega(G,S)=\alpha$, so the statement is proved.
\end{proof}

Note that since the given presentation of $\mathrm{PGL}(2,\mathbb{Z})$ with respect to the set $S=\{a,b,c\}$ is a presentation of a Coxeter group, its growth function and hence its exponential growth rate $\omega(G,S)$ can easily be deduced from Steinberg's recursive formula for growth functions of Coxeter groups \cite[Corollary 1.29]{Steinberg}.

\medskip

We may go a little further and prove that for the group $\tilde{G}=\mathrm{GL}(2,\mathbb{Z})=\mathrm{Aut}(\mathbb{F}_2)$ we also have  $\Omega(\tilde{G})=\alpha$. Indeed, $\tilde{G}$ is a central extension of $G$ by a cyclic group, and the group $G=\mathrm{PGL}(2,\mathbb{Z})$ has a linear Dehn function since it is an amalgamated free product of finite groups. Then we can use \cite[Lemma 12]{Talambutsa2011} to conclude that $\Omega(\tilde{G})=\Omega(G)=\alpha$.

Using the following classical presentation \cite[formula 7.21]{Coxeter-Moser}:
\begin{equation}
\tilde{G} = \langle a,b,c, z \mid (ab)^2=(bc)^3=z, z^2=1, a^2=b^2=c^2=1 \rangle.
\label{GL-presentation}
\end{equation}
one can show that $\mathrm{GL(2,\mathbb{Z})}$ can be decomposed as $D_8 *_{(C_2\times C_2)} (D_{12})$. Indeed, introducing a new generator $d=b$ and substituting $z$ by $(ab)^2$ we get
$$
\tilde{G}=\langle a,b,c,d \mid d=b, (ab)^4=1, (ab)^2=(cd)^3, a^2=b^2=c^2=d^2=1 \rangle.
$$
This presentation can be obtained taking a formal amalgamation of the groups $D_8=\langle a,b \mid a^2=b^2=1, (ab)^4=1$ and $D_{12}=\langle c,d \mid c^2=d^2=1, (cd)^{6}=1\rangle$ with isomorphism sending central element $(ab)^2$ to $(cd)^3$ and $b$ to $d$. Both amalgamated subgroups $\langle (ab)^2, b \rangle$ and $\langle (cd)^3, d \rangle$ are isomorphic to $C_2 \times C_2$.

\smallskip

If we modify the presentation \eqref{GL-presentation}, removing the relation $z^2=1$ we will get another central extension $\overline{G}$ of the group $\mathrm{PGL}(2,\mathbb{Z})$ by the subgroup $\langle z \rangle \cong \mathbb{Z}$. So the same argument as above shows that $\Omega(\overline{G})=\alpha$. The group $\overline{G}$ can be obtained as an amalgamated product of two infinite dihedral groups $\langle a,b \mid a^2=b^2=1\rangle$ and $\langle c^2=d^2=1 \rangle $ by the isomorphism $(ab)^2 \to (cd)^3$ and $b\to d$. The group $\overline{G}$ can be considered as the analogue of the trefoil knot group $\langle a,b \mid a^2=b^3\rangle$ that is the universal central extension for the groups $\mathrm{PSL}(2,\mathbb{Z})$ and $\mathrm{SL}(2,\mathbb{Z})$.

\medskip

Finally, notice that if for some number $\beta$ there is an example of amalgamated free product $G=A *_C B$ such that $\Omega(G)=\beta$, then there are infinitely many such examples. Such series can be constructed using two simple observations. First, a direct product $P\times (A *_C B)$ can also be presented as $(A\times P) *_{C\times P} (B\times P)$. Second, we have $\Omega(P\times G)=\max(\Omega(P),\Omega(G))$ (see \cite[Proposition 5]{Mann}). Thus, taking $G_n=\mathbb{Z}^n \times G$ we will get a series of groups which are not isomorphic because they have different abelianizations and such that $\Omega(G_n)=\Omega(G)=\beta$.

\bibliographystyle{amsalpha}

\end{document}